\documentclass[a4paper]{amsart}

\usepackage{amsmath,amssymb,amsthm}
\usepackage[all]{xy}
\usepackage[dvipdfmx,colorlinks=true]{hyperref}

\newtheorem{theorem}{Theorem}[section]
\newtheorem{proposition}[theorem]{Proposition}
\newtheorem{lemma}[theorem]{Lemma}
\newtheorem{corollary}[theorem]{Corollary}

\theoremstyle{definition}

\theoremstyle{remark}

\numberwithin{equation}{section}

\newcommand{\Z}{\mathbb{Z}}
\newcommand{\Q}{\mathbb{Q}}

\newcommand{\C}{\mathbb{C}}
\newcommand{\honil}{\operatorname{honil}}

\SelectTips{cm}{}

\title[Homotopy commutativity in Hermitian symmetric spaces]{Homotopy commutativity in Hermitian symmetric spaces}

\author[Daisuke Kishimoto]{Daisuke Kishimoto}
\address{Department of Mathematics, Kyoto University, Kyoto, 606-8502, Japan}
\email{kishi@math.kyoto-u.ac.jp}

\author[Masahiro Takeda]{Masahiro Takeda}
\address{Department of Mathematics, Kyoto University, Kyoto, 606-8502, Japan}
\email{takeda.masahiro.87u@st.kyoto-u.ac.jp}

\author[Yichen Tong]{Yichen Tong}
\address{Department of Mathematics, Kyoto University, Kyoto, 606-8502, Japan}
\email{tong.yichen.25m@st.kyoto-u.ac.jp}

\date{\today}

\subjclass[2010]{55P35, 55Q15}

\keywords{homotopy commutativity, Hermitian symmetric space, flag manifold, Samelson product, Whitehead product}

\begin{document}

\maketitle

\begin{abstract}
  Ganea proved that the loop space of $\C P^n$ is homotopy commutative if and only if $n=3$. We generalize this result to that the loop spaces of all irreducible Hermitian symmetric spaces but $\C P^3$ are not homotopy commutative. The computation also applies to determining the homotopy nilpotency of the loop spaces of flag manifolds.
\end{abstract}


\section{Introduction}\label{introduction}

A fundamental problem on H-spaces is to find whether or not a given H-space is homotopy commutative. This was intensely studied for finite H-spaces, and a complete answer was given by Hubbuck \cite{H} such that if a connected finite H-space is homotopy commutative, then it is homotopy equivalent to a torus. As for infinite H-spaces, the problem should be studied by fixing a class of infinite H-spaces because there are too many classes of infinite H-spaces, each of which has its own special features.

In \cite{G}, Ganea studied the homotopy nilpotency of complex projective spaces, and in particular, he proved that the loop space of the complex projective space $\C P^n$ is homotopy commutative if and only if $n=3$. Then we continue this work to study the homotopy commutativity of the loop spaces of homogeneous spaces. Recently, Golasi\'{n}ski \cite{Go} showed that the loop spaces of some homogeneous spaces such as complex Grassmannians are homotopy nilpotent. However, their homotopy nilpotency classes are not computed: it is not even proved that they are homotopy commutative or not. In this paper, we study the homotopy commutativity of Hermitian symmetric spaces, which generalizes Ganea's result and makes Golasi\'{n}ski's result more concrete. Recall that every Hermitian symmetric space is a product of irreducible ones in the following table.

\renewcommand{\arraystretch}{1.2}

\begin{table}[htbp]
  \centering
  \begin{tabular}{l|ll}
    \hline
    AIII&$U(m+n)/U(m)\times U(n)$&$(m,n\ge 1)$\\
    BDI&$SO(n+2)/SO(2)\times SO(n)$&$(n\ge 3)$\\
    CI&$Sp(n)/U(n)$&$(n\ge 4)$\\
    DIII&$SO(2n)/U(n)$&$(n\ge 4)$\\
    EIII&$E_6/Spin(10)\cdot T^1$&$(Spin(10)\cap T^1\cong\Z/4)$\\
    EVII&$E_7/E_6\cdot T^1$&$(E_6\cap T^1\cong\Z/3)$\\\hline
  \end{tabular}
\end{table}

\noindent Then we only need to consider the loop spaces of irreducible Hermitian symmetric spaces. Now we state the main theorem.

\begin{theorem}
  \label{main}
  The loop spaces of all irreducible Hermitian symmetric spaces but $\C P^3$ are not homotopy commutative.
\end{theorem}

Theorem \ref{main} will be proved by a case-by-case analysis of irreducible Hermitian symmetric spaces. Our main tools for the analysis are rational homotopy theory (Section \ref{rational homotopy}) and Steenrod operations (Section \ref{Steenrod operation}). The rational homotopy technique also applies to flag manifolds, so that we can prove the following, where the definition of the homotopy nilpotency will be given in Section \ref{rational homotopy}.

\begin{theorem}
  \label{main flag}
  Let $G$ be a compact connected non-trivial Lie group with maximal torus $T$. Then the loop space of the flag manifold $G/T$ is homotopy nilpotent of class 2.
\end{theorem}


\subsection*{Acknowledgement}

The authors are grateful to Toshiyuki Miyauchi for informing them the result of \={O}shima \cite{O}. The first author were partially supported by JSPS KAKENHI Grant Numbers 17K05248 and 19K03473 (Kishimoto), JSPS KAKENHI Grant Number 21J10117 (Takeda), and JST SPRING Grant Number JPMJSP2110 (Tong).


\section{Rational homotopy}\label{rational homotopy}

In this section, we apply rational homotopy theory to prove that the loop spaces of irreducible Hermitian symmetric spaces of type CI, DIII and EVII are not homotopy commutative. We also consider the homotopy nilpotency of flag manifolds. By \cite[Proposition 13.16]{FHT} and the adjointness of Whitehead products and Samelson products, we have the following criterion for a loop space not being homotopy commutative.

\begin{lemma}
  \label{d_2}
  Let $(\Lambda V,d)$ be the minimal Sullivan model of a simply-connected CW complex of finite type $X$. If there is $x\in V$ such that
  \[
    dx\not\equiv 0\mod\Lambda^3V,
  \]
  then $\Omega X$ is not homotopy commutative.
\end{lemma}

In order to apply Lemma \ref{d_2}, we will use the following lemma.

\begin{lemma}
  \label{model}
  Let $X,Y$ be simply-connected spaces such that
  \[
    H^*(X;\Q)=\Q[x_1,\ldots,x_m]\quad\text{and}\quad H^*(Y;\Q)=\Q[y_1,\ldots,y_n].
  \]
  If a map $f\colon X\to Y$ is injective in rational cohomology, then there is a Sullivan model of the homotopy fiber of $f$ such that
  \[
    (\Lambda(x_1,\ldots,x_m,z_1,\ldots,z_n),d),\quad dx_i=0,\quad dz_i=f^*(y_i).
  \]
\end{lemma}

\begin{proof}
  By the Borel transgression theorem, $H^*(\Omega Y;\Q)=E(z_1,\ldots,z_n)$ such that $\tau(z_i)=y_i$, where $E(z_1,\ldots,z_n)$ denotes the exterior algebra generated by $z_1,\ldots,z_n$ and $\tau$ denotes the transgression. Let $F$ denote the homotopy fiber of the map $f$. Then the sequence
  \[
    (\Lambda(x_1,\ldots,x_m),0)\xrightarrow{\rm incl}(\Lambda(x_1,\ldots,x_m,z_1,\ldots,z_n),d)\xrightarrow{\rm proj}(\Lambda(z_1,\ldots,z_n),0)
  \]
  is a model of the principal fibration $\Omega Y\to F\to X$, where $dx_i=0$ and $dz_i=f^*(y_i)$. Thus the statement is proved.
\end{proof}

\begin{proposition}
  \label{CD}
  The loop spaces of $Sp(n)/U(n)$ and $SO(2n)/U(n)$ are not homotopy commutative.
\end{proposition}

\begin{proof}
  First, we consider $Sp(n)/U(n)$. Recall that the cohomology of $BU(n)$ and $BSp(n)$ are given by
  \[
    H^*(BU(n);\Z)=\Z[c_1,\ldots,c_n]\quad\text{and}\quad H^*(BSp(n);\Z)=\Z[q_1,\ldots,q_n],
  \]
  where $c_i$ and $q_i$ are the Chern classes and the symplectic Pontrjagin classes. Then as in \cite[Chapter III, Theorem 5.8]{MT}, the natural map $q\colon BU(n)\to BSp(n)$ satisfies
  \[
    q^*(q_i)=\sum_{k+l=i}(-1)^{i+k}c_kc_l,
  \]
  where $c_0=1$ and $c_i=0$ for $i>n$. Then by Lemma \ref{model}, there is a Sullivan model of $Sp(n)/U(n)$ such that
  \[
    (\Lambda(c_1,\ldots,c_n,r_{1},\ldots r_n),d),\quad dc_i=0,\quad dr_i=\sum_{k+l=2i}(-1)^{i+k}c_kc_l,
  \]
  where $c_0=1$ and $c_i=0$ for $i>n$. Hence the minimal model of $Sp(n)/U(n)$ is given by
  \begin{gather*}
    (\Lambda(c_1,c_3,\ldots,c_{2n-2[n/2]-1},r_{[n/2]+1},\ldots r_n),d)\\
    dc_i=0,\quad dr_i\equiv\sum_{k+l=i-1}(-1)^{i+1}c_{2k+1}c_{2l+1}\mod(c_1,c_3,\ldots,c_{2n-2[n/2]-1})^4,
  \end{gather*}
  where $c_0=1$, $c_i=0$ for $i>n$. Thus modulo $(c_1,c_3,\ldots,c_{2n-2[n/2]-1})^4$,
  \[
    dr_{n-1}\equiv c_{n-1}^2\quad(n\text{ is even}),\quad dr_n\equiv c_n^2\quad(n\text{ is odd}).
  \]
  Therefore, by Lemma \ref{d_2}, $\Omega(Sp(n)/U(n))$ is not homotopy commutative.

  Next, we consider $SO(2n)/U(n)$. The rational cohomology of $BSO(2n)$ is given by
  \[
    H^*(BSO(2n);\Q)=\Q[p_1,\ldots,p_{n-1},e],
  \]
  where $p_i$ are the $i$-the Pontrjagin classes and $e$ is the Euler class. By \cite[Lemma 5.15 and Theorem 5.17]{MT}. Then the natural map $r\colon BU(n)\to BSO(n)$ satisfies
  \[
    r^*(p_i)=\sum_{k+l=2i}(-1)^kc_kc_l\quad\text{and}\quad r^*(e)=c_n,
  \]
  where $c_0=1$ and $c_i=0$ for $i>n$. Thus arguing as above, we can see that the minimal model of $SO(2n)/U(n)$ coincides with that of $Sp(n-1)/U(n-1)$, implying that $\Omega(SO(2n)/U(n))$ is not homotopy commutative.
\end{proof}

\begin{proposition}
  \label{EVII}
  The loop space of $E_7/E_6\cdot T^1$ is not homotopy commutative.
\end{proposition}

\begin{proof}
  As in the proof of \cite[Lemma 2.1]{W}, we have
  \begin{align*}
    H^*(BE_7;\Q)&=\Q[x_4,x_{12},x_{16},x_{20},x_{24},x_{28},x_{36}]\\
    H^*(B(E_6\cdot T^1);\Q)&=\Q[u,v,w,x_4,x_{12},x_{16},x_{24}],
  \end{align*}
  where $|x_i|=i$, $|u|=2$, $|v|=10$ and $|w|=18$. Moreover, the natural map $j\colon B(E_6\cdot T^1)\to BE_7$ satisfies $j^*(x_i)=x_i$ for $i=4,12,16,24$ and $j^*(x_i)\equiv z_i\mod(x_4,x_{12},x_{16},x_{24})$ for $i=20,28,36$, where
  \begin{alignat*}{2}
    z_{20}&=v^2-2uv&z_{28}&=-2vw+18u^5w-6u^6v+u^{14}\\
    z_{36}&=w^2+20u^4vw-18u^9w+2u^{13}v.\quad
  \end{alignat*}
  Then by Lemma \ref{model}, there is a Sullivan model of $E_7/E_6\cdot T^1$ such that
  \[
    (\Lambda(u,v,w,x_4,x_{12},x_{16},x_{24},y_3,y_{11},y_{15},y_{19},y_{23},y_{27},y_{36}),d),
  \]
  where $du=dv=dw=0$ and $dy_i=x_{i+1}$ for $i=3,11,15,23$ and $dy_i\equiv z_{i+1}\mod(x_4,x_{12},x_{16},x_{24})$. Thus we can easily see that the minimal model of $E_7/E_6\cdot T^1$ is given by $(\Lambda(u,v,w,y_{19},y_{27},y_{36}),d)$ such that $du=dv=dw=0$ and $dy_i=z_{i+1}$ for $i=19,27,36$. Therefore by Lemma \ref{d_2}, $\Omega(E_7/E_6\cdot T^1)$ is not homotopy commutative as stated.
\end{proof}

We consider the homotopy nilpotency of flag manifolds. Let $X$ be an H-group. Let $\gamma\colon X\wedge X\to X$ denote the reduced commutator map, and let $\gamma_n=\gamma\circ(\gamma_{n-1}\wedge 1_X)$ for $n\ge 2$ and $\gamma_1=1_X$. Recall from \cite[Definition 2.6.2]{Z} that $X$ is called \emph{homotopy nilpotent of class $<n$} if $\gamma_n\simeq*$. Let $\honil(X)$ denote the homotopy nilpotency class of $X$. Then $X$ is homotopy commutative if and only if $\honil(X)\le 1$. By \cite[Lemma 2.6.6]{Z}, we have:

\begin{proposition}
  \label{nilpotency extension}
  Let $f\colon X\to Y$ be an H-map between H-groups with homotopy fiber $F$. Then
  \[
    \honil(F)\le\honil(X)+1.
  \]
\end{proposition}

\begin{corollary}
  \label{honil}
  Let $G$ be a topological group, and let $H$ be a subgroup of $G$. Then
  \[
    \honil(\Omega(G/H))\le\honil(H)+1.
  \]
\end{corollary}

\begin{proof}
  The homotopy fiber of the inclusion $H\to G$ is $\Omega(G/H)$, and so by Proposition \ref{nilpotency extension}, the proof is finished.
\end{proof}

Hopkins \cite[Corollary 2.2]{Ho} proved that a connected finite H-group is homotopy nilpotent whenever it is torsion free in homology. Then for a compact connected Lie group $G$ and its closed subgroup $H$, it follows from Corollary \ref{honil} that $\Omega(G/H)$ is homotopy nilpotent whenever $H$ is torsion free in homology (cf. \cite[Proposition 2.2]{Go}).  In particular, we obtain that the loop space of the flag manifold $G/T$ is homotopy nilpotent, where $T$ is a maximal torus of $G$. Now we are ready to prove Theorem \ref{main flag}.

\begin{proof}
  [Proof of Theorem \ref{main flag}]
  Clearly, we may assume $G$ is simply-connected. Since $T$ is homotopy commutative and non-contractible, we have $\honil(T)=1$. Then by Corollary \ref{honil}, $\honil(\Omega(G/T))\le 2$, and so it remains to show that $\Omega(G/T)$ is not homotopy commutative. It is well known that the natural map $H^*(BG;\Q)\to H^*(BT;\Q)^W$ is an isomorphism and
  \[
    H^*(BT;\Q)^W=\Q[x_1,\ldots,x_n],
  \]
  where $W$ is the Weyl group of $G$. Since $G$ is simply-connected, $H^*(BG;\Q)=0$ for $*\le 3$ and $H^4(BG;\Q)\ne 0$. Then we may assume $|x_1|=4$. By Lemma \ref{model}, there is a Sullivan model of $G/T$ such that
  \[
    (\Lambda(t_1,\ldots,t_n,y_1,\ldots,y_n),d),\quad dt_i=0,\quad dy_i=x_i,
  \]
  where $t_1,\ldots,t_n$ are generators of $H^*(BT;\Q)$ which are of degree 2. Since all $x_i$ are decomposables by degree reasons, this is the minimal model of $G/T$. Moreover, $x_1$ is a quadratic polynomial in $t_1,\ldots,t_n$. Then by Lemma \ref{d_2}, $G/T$ has non-trivial Whitehead product, implying that $\Omega(G/T)$ is not homotopy commutative.
\end{proof}


\section{Steenrod operation}\label{Steenrod operation}

In this section, we prove that the loop spaces of the irreducible Hermitian symmetric spaces of type AIII, BDI, EIII by applying the following lemma. The lemma was proved by Kono and \={O}shima \cite{KO} when $A$ and $B$ are spheres and $p$ is odd, and its variants are used in \cite{HK1,HK2,HKMO,HKO,KK,KM,KOT,KT,T}. For an augmented graded algebra $A$, let $QA^n$ denote the module of indecomposables of dimension $n$.

\begin{lemma}
  \label{criterion}
  Let $X$ be a path-connected space $X$, let $\alpha\colon\Sigma A\to X,\,\beta\colon\Sigma B\to X$ be maps, and let $p$ be a prime. Suppose the following conditions hold:
  \begin{enumerate}
    \item there are $a,b\in H^*(X;\Z/p)$ such that $\alpha^*(a)\ne 0$, $\beta^*(b)\ne 0$, and
    \begin{enumerate}
      \item $\alpha^*(b)=0$ or $\beta^*(a)=0$ for $p=2$,

      \item $A=B$, $\alpha=\beta$ and $a=b$ for $|a|=|b|$ and $p$ odd;
    \end{enumerate}

    \item there are $x\in H^*(X;\Z/p)$ and a Steenrod operation $\theta$ such that $\theta(x)$ is decomposable and includes the term $ab\ne 0$;

    \item $\dim QH^{*}(X;\Z/p)=1$ for $*=|a|,|b|$;

    \item $\theta$ acts trivially on $H^*(\Sigma A\times\Sigma B;\Z/p)$.
  \end{enumerate}
  Then the Whitehead product $[\alpha,\beta]$ in $X$ is non-trivial.
\end{lemma}

\begin{proof}
  Suppose $[\alpha,\beta]=0$. Then there is a homotopy commutative diagram
  \[
    \xymatrix{
      \Sigma A\vee\Sigma B\ar[r]^(.65){\alpha\vee\beta}\ar[d]_{\rm incl}&X\ar@{=}[d]\\
      \Sigma A\times\Sigma B\ar[r]^(.65)\mu&X.
    }
  \]
  By the conditions (1), (2) and (3), the $H^{|a|}(\Sigma A;\Z/p)\otimes H^{|b|}(\Sigma B;\Z/p)$-part of $\mu^*(\theta(x))$ is
  \begin{align*}
    \mu^*(ab)&=(\alpha^*(a)\times 1+1\times\beta^*(a))(\alpha^*(b)\times 1+1\times\beta^*(b))\\
    &=
    \begin{cases}
      2\alpha^*(a)\times\beta^*(b)&|a|=|b|\text{ and }p\text{ odd}\\
      \alpha^*(a)\times\beta^*(b)&\text{otherwise},
    \end{cases}
  \end{align*}
  implying $\mu^*(\theta(x))\ne 0$. By the condition (4), we have $\mu^*(\theta(x))=\theta(\mu^*(x))=0$. Then we obtain a contradiction, implying $[\alpha,\beta]\ne 0$, as stated.
\end{proof}

Let $G_{m,n}=U(m+n)/U(m)\times U(n)$. Since $G_{m,n}\cong G_{n,m}$, we may assume $m\le n$. Let $j\colon G_{m,n}\to BU(m)$ denote the natural map. Then since $m\le n$, the map $j$ is a $(2m+1)$-equivalence. Let $g_i\colon S^{2i}\to BU(m)$ denote a generator of $\pi_{2i}(BU(m))\cong\Z$ for $i=1,\ldots,m$. Then since $j$ is a $(2m+1)$-equivalence, there is a map $\bar{g}_i\colon S^{2i}\to G_{m,n}$ such that $j\circ\bar{g}_i=g_i$ for each $i\le m$. Thus
\[
  j\circ[\bar{g}_k,\bar{g}_l]=[j\circ\bar{g}_k,j\circ\bar{g}_l]=[g_k,g_l].
\]
So if $[g_k,g_l]\ne 0$, then $[\bar{g}_k,\bar{g}_l]\ne 0$, implying that $\Omega G_{m,n}$ is not homotopy commutative. We can find a non-trivial Whitehead product $[g_k,g_l]$ by using the result of Bott \cite{B}, but here we use Lemma \ref{criterion} instead.

Recall from \cite[Chapter III, Theorem 6.9]{MT} that the cohomology of $G_{m,n}$ is given by
\[
  H^*(G_{m,n};\Z)=\Z[c_1,\ldots,c_m,\bar{c}_1,\ldots,\bar{c}_n]/(\sum_{i+j=k}c_i\bar{c}_j\mid k\ge 1)
\]
such that $j^*(c_i)=c_i$ for each $i$, where $c_0=\bar{c}_0=1$, $c_i=0$ for $i>m$, $\bar{c}_j=0$ for $j>n$ and the cohomology of $BU(m)$ is as in the proof of Proposition \ref{CD}. We say that a cohomology class $x\in H^k(X;\Z/p)$ is \emph{mod $p$ spherical} if there is a map $\alpha\colon S^k\to X$ such that $\alpha^*(x)\ne 0$. We denote the mod $p$ reduction of an integral cohomology class by the same symbol $x$.

\begin{lemma}
  \label{spherical}
  If $p$ is a prime, then $c_i$ is mod $p$ spherical for $i\le p$.
\end{lemma}

\begin{proof}
  By \cite[Chapter IV, Lemma 5.8]{MT}, $g_i^*(c_i)=\pm(i-1)!u_{2i}$, where $u_{2i}$ is a generator of $H^{2i}(S^{2i};\Z)\cong\Z$. Then the proof is done.
\end{proof}

\begin{proposition}
  \label{AIII}
  The loop space of $G_{m,n}$ for $m,n\ge 2$ is not homotopy commutative.
\end{proposition}

\begin{proof}
  As observed above, it suffices to show $[g_k,g_l]\ne 0$ for some $k,l$. First, we consider the $m=2$ case. By Lemma \ref{spherical}, $c_1,c_2\in H^*(G_{2,n};\Z/2)$ are mod 2 spherical. By the Wu formula, $\mathrm{Sq}^2c_2=c_1c_2\ne 0$ in $H^*(BU(2);\Z/2)$. Then by Lemmas \ref{criterion} and \ref{spherical}, $[g_1,g_2]\ne 0$.

  Next, we consider the $m>2$ case. Take any odd prime $p$ with $m/2<p\le m$, where such an odd prime exists by Bertrand's postulate. Let $k=m/2$ for $m$ even and $k=(m+1)/2$ for $m$ odd. By Lemma \ref{spherical}, $c_k$ and $c_{m-k+1}$ are mod $p$ spherical. By the mod $p$ Wu formula proved by Shay \cite{Sh}, $\mathcal{P}^1c_{m-p+2}$ is decomposable and includes the term
  \[
    -(m+1)c_kc_{m-k+1}
  \]
  in $H^*(BU(m);\Z/p)$. So if $m+1\not\equiv 0\mod p$, then $[g_k,g_{m-k+1}]\ne 0$. Now we suppose $m+1\equiv 0\mod p$. Then we must have $m=2p-1$. So if there is another prime $q$ in $(m/2,m]$, then $m+1\not\equiv 0\mod q$. So the above argument for the $m+1\not\equiv 0\mod p$ case works, and thus $[g_k,g_{m-k+1}]\ne 0$. Hence we aim to show that there are two primes in $(m/2,m]$. Recall from \cite{So} that the Ramanujan prime $R_n$ is the least integer $k$ such that for each $x\ge k$, there are at least $n$ primes in the interval $(x/2,x]$. It is proved in \cite{So} that $R_n$ exists for each $n$ and $R_2=11$. Then it remains the cases $m=2\cdot 3-1=5$ and $m=2\cdot 5-1=9$, and we have $5/2<3,5\le 5$ and $9/2<5,7\le 9$. Thus there are at least two primes in $(m/2,m]$, completing the proof.
\end{proof}

Let $Q_n=SO(n+2)/SO(2)\times SO(n)$.

\begin{proposition}
  \label{BDI}
  The loop space of $Q_n$ for $n\ge 2$ is not homotopy commutative.
\end{proposition}

\begin{proof}
  There is a homotopy fibration
  \begin{equation}
    \label{Q fibration}
    S^1=SO(2)\to SO(n+2)/SO(n)\xrightarrow{q}Q_n.
  \end{equation}
  Then the projection $q\colon SO(n+2)/SO(n)\to Q_n$ is injective in $\pi_*$ for $*\ge 2$, and so by the naturality of Whitehead products, it is sufficient to show that there is a non-trivial Whitehead products in $\pi_*(SO(n+2)/SO(n))$ for some $*\ge 2$. Let $\iota\colon S^n=SO(n+1)/SO(n)\to SO(n+2)/SO(n)$ denote the inclusion. Then \={O}shima \cite{O} proved that the Whitehead product $[\iota,\iota]\in\pi_{2n-1}(SO(n+2)/SO(n))$ is non-trivial whenever $n+1$ is not the power of 2. Thus we obtain that $\Omega Q_n$ is not homotopy commutative if $n+1$ is not the power of 2.

  Suppose $n=2m-1$. Then as in \cite{I}, the cohomology of $Q_n$ is given by
  \[
    H^*(Q_n;\Z)=\Z[t,e]/(t^{m}-2e,e^2),\quad\mathrm{Sq}^2e=te,
  \]
  where $|t|=2$ and $|e|=2m$. Since $Q_n$ is simply-connected, the Hurewicz theorem implies that $t$ is mod 2 spherical. Let $B=S^{n-1}\cup_2e^n$. Then $SO(n+2)/SO(n)=\Sigma B\cup e^{2n+1}$, so that
  \[
    H^*(SO(n+2)/SO(n);\Z/2)=E(x_n,x_{n+1}),\quad|x_i|=i.
  \]
  Let $j\colon\Sigma B\to Q_n$ denote the composition of the inclusion $\Sigma B\to SO(n+2)/SO(n)$ and the projection $q\colon SO(n+2)/SO(n)\to Q_n$. Then by the Gysin sequence for the fibration \eqref{Q fibration}, we get $j^*(e)=x_{n+1}$. Thus by Lemma \ref{criterion}, we obtain that $Q_n$ has non-trivial Whitehead product, implying $\Omega Q_n$ is not homotopy commutative.
\end{proof}

\begin{proposition}
  \label{EIII}
  The loop space of $E_6/Spin(10)\cdot T^1$ is not homotopy commutative.
\end{proposition}

\begin{proof}
  As in \cite{I}, the mod 2 cohomology of $E_6/Spin(10)\cdot T^1$ is given by
  \[
    H^*(E_6/Spin(10)\cdot T^1;\Z/2)=\Z/2[t,w']/(tw'^2,t^{12}+w'^3),\quad \mathrm{Sq}^2w'=tw',
  \]
  where $|t|=2$ and $|w'|=8$. Since $E_6/Spin(10)\cdot T^1$ is simply-connected, the Hurewicz theorem implies that $t$ is mod 2 spherical. We can deduce from Conlon's result \cite{C} that $\pi_*(E_6/Spin(10),F_4/Spin(9))=0$ for $*\le 31$. In particular,
  \[
    H^*(E_6/Spin(10);\Z/2)\cong H^*(F_4/Spin(9);\Z/2)\quad(*\le 30).
  \]
  Note that $F_4/Spin(9)$ is the Cayley plane $\mathbb{O}P^2$. Then since $\mathbb{O}P^2=S^8\cup e^{16}$, a generator $u\in H^8(F_4/Spin(9);\Z/2)\cong\Z/2$ is mod 2 spherical, and so a generator $v\in H^8(E_6/Spin(10))\cong\Z/2$ is mod 2 spherical too. By the Gysin sequence associated to the fibration $S^1\to E_6/Spin(10)\xrightarrow{q}E_6/Spin(10)\cdot T^1$, we can see that $q^*(w')=v$, implying $w'$ is mod 2 spherical. Thus by Lemma \ref{criterion}, we obtain that $E_6/Spin(10)\cdot T^1$ has a non-trivial Whitehead product, and so $\Omega(E_6/Spin(10)\cdot T^1)$ is not homotopy commutative.
\end{proof}

Now we are ready to prove Theorem \ref{main}.

\begin{proof}
  [Proof of Theorem \ref{main}]
  Combine Propositions \ref{CD}, \ref{EVII}, \ref{AIII}, \ref{BDI}, \ref{EIII} and the result of Ganea \cite{G} on the homotopy commutativity of the loop space of $\C P^n$ mentioned in Section \ref{introduction}.
\end{proof}

\end{document}